\documentclass[preprint,12pt]{elsarticle}
\usepackage{amsthm}
\usepackage{amssymb,amsmath}
\usepackage{pstricks}
\usepackage{ifthen}
\usepackage{calc}

\theoremstyle{plain}
\newtheorem{thm}{Theorem}
\newtheorem{lem}[thm]{Lemma}
\newtheorem{prop}[thm]{Proposition}

\theoremstyle{definition}

\newtheorem{rem}{Remark}

\newproof{pf}{Proof}
\newproof{pot}{Proof of Theorem \ref{Theoremeprincipal}}
\newproof{pop}{Proof of Proposition \ref{key coeff non nul}}

 \makeatletter
\@addtoreset{thm}{chapter}
\@addtoreset{thm}{section}
\@addtoreset{thm}{subsection}

\makeatother

\begin{document}

\begin{frontmatter}

\title{A matrix description of weakly bipartitive and bipartitive families }

\author[B]{Edward Bankoussou-mabiala }
\ead{bankoussoumabiala@yahoo.fr}

\author[B]{Abderrahim  Boussa\"{\i}ri\corref{cor1}}
\ead{aboussairi@hotmail.com}

\author[B]{\\Abdelhak Cha\"{\i}cha\^{a}}
\ead{abdelchaichaa@gmail.com}

\author[B]{Brahim Chergui}
\ead{cherguibrahim@gmail.com}

\cortext[cor1]{Corresponding author}
%\cortext[cor2]{Principal corresponding author}

\address[B]{Facult\'e des Sciences A\"{\i}n Chock, D\'epartement de
Math\'ematiques et Informatique  Laboratoire de Topologie, Alg\`{e}bre, G\'{e}om\'{e}trie et Math\'{e}matiques discr\`{e}tes

Km 8 route d'El Jadida,
BP 5366 Maarif, Casablanca, Maroc}

\begin{abstract}
The notions of weakly bipartitive and bipartitive families were introduced by
Montgolfier (2003) as a general tool for studying some decomposition of graphs
and other combinatorial structures. In this paper, we give a matrix
description of these notions.

\end{abstract}

\begin{keyword}
Graph; Modular decomposition; Bipartitive families; matrices.

\MSC 	05C20, 15A03
\end{keyword}

\end{frontmatter}

%%%%%%%%%%%%%%%%%%%%%%%%%%%%%%%%%%%%%%%%%%%%%%%%%%%%%%%%%%%%%%%%%%%%%%%%%%%%%%%%%%%%%
%%%%%%%%%%%%%%%%%%%%%%%%%%%%%%%%%%%    Introduction       %%%%%%%%%%%%%%%%%%%%%%%%%%%%%%%%%%%%%%%%%
%%%%%%%%%%%%%%%%%%%%%%%%%%%%%%%%%%%%%%%%%%%%%%%%%%%%%%%%%%%%%%%%%%%%%%%%%%%%%%%%%%%%%%

\section{Introduction}

Modular decomposition has arisen as a technique that applies to many
combinatorial structures such as graphs, tournaments, 2-structures,
hypergraphs, and matroids, among others. It is based on module. For graphs, this notion goes back to Gallai \cite{ Gallai67}. More precisely, let $G=(V,E)$ be an undirected simple
graph. A \emph{module} of $G$ is a set $M\subseteq V$ such that for all $x$
$\in V\setminus M$ either $N_{G}(x)\cap M=\emptyset$ or $M\subseteq N_{G}(x)$,
where $N_{G}(x)$ is the \emph{neighborhood} of $x$, that is, $N_{G}%
(x):=\left \{  y\in V:\left \{  x,y\right \}  \in E\right \}  $. For tournaments,
the notion of module can be defined in a similar way. Recall that a
\emph{tournament} is a directed graph such that for every distinct vertices
$x$ and $y$, either $x\longrightarrow y$ or $y\longrightarrow x$ and never
both. Let $T$ be a tournament with vertex set $V$. The \emph{out-neighborhood}
of a vertex $x\in V$ is the set $N_{T}^{+}(x)=\{y\in V:x\longrightarrow y\}$
and the \emph{in-neighborhood} is $N_{T}^{-}(x)=\{y\in V:y\longrightarrow
x\}$. A module of $T$ is a set $M$ $\subseteq V$ such that for all $x$ $\in
V\setminus M$ either $N_{T}^{+}(x)\cap M=\emptyset$ or $M\subseteq N_{T}%
^{+}(x)$.

The split decomposition\ of graphs and the bi-join decomposition of graphs and
of tournaments can be seen as a generalization of the modular decomposition.
These decompositions were introduced respectively by Cunningham \cite{Cun82}
and Montgolfier \cite{Mon03}. Let $G=(V,E)$ be an undirected simple graph and
let $\left \{  X,Y\right \}  $ be a bipartition of $V$. We say that $\left \{
X,Y\right \}  $ is a \emph{split} of $G$ if there exist $X_{1}\subseteq X$ and
$Y_{1}\subseteq Y$ such that for all $x\in X_{1}$, $N_{G}(x)\cap Y=Y_{1}$ and
for all $x\in X\setminus X_{1}$, $N_{G}(x)\cap Y=\emptyset$. We say that
$\left \{  X,Y\right \}  $ is a \emph{bi-join }of $G$ if there exist
$X_{1}\subseteq X$ and $Y_{1}\subseteq Y$ such that for all $x\in X_{1}$,
$N_{G}(x)\cap Y=Y_{1}$ and for all $x\in X\setminus X_{1}$, $N_{G}(x)\cap$ $Y$
$=Y\setminus Y_{1}$. Remark that if $X$ or $Y$ is a module of $G$ then
$\left \{  X,Y\right \}  $ is both a split \ and a bi-join of $G$. The notion of
bi-join can be also defined for tournaments in the following way.
Let $T$ be a tournament with vertex set $V$. A bipartition $\left \{
X,Y\right \}  $ of $V$ is a \emph{bi-join }of $T$ if there exist $X_{1}%
\subseteq X$ and $Y_{1}\subseteq Y$ such that for all $x\in X_{1}$ (resp.
$x\in X\setminus X_{1}$), $N_{T}^{+}(x)\cap Y=Y_{1}$ and $N_{T}^{-}(x)\cap
Y=Y\setminus Y_{1}$ (resp. $N_{T}^{+}(x)\cap Y=Y\setminus Y_{1}$ and
$N_{T}^{-}(x)\cap$ $Y$ $=Y_{1}$).

%%%%%%%%%%%%%%%%%%%%%%%%%%%%%%%%%%%%%%%%%%%%%%%%%%%%%%%%%%%%%%%%%%%%%%%%%%%%%%%%%%%%%%%%%%
%%%%%%%%%%%%%%%%%%%%%%%%%%%%%%%   Fig 1   A split in a graph   %%%%%%%%%%%%%%%%%%%%%%%%%%%%%%%%%%%%%
%%%%%%%%%%%%%%%%%%%%%%%%%%%%%%%%%%%%%%%%%%%%%%%%%%%%%%%%%%%%%%%%%%%%%%%%%%%%%%%%%%%%%%%%%%%%
\begin{figure}[h]
\begin{center}
\setlength{\unitlength}{0.7cm}
\begin{picture}(30,2.6)
\thicklines
\put(8,1){\oval(2,1)}
\put(7.1,0.8){$X\setminus X_{1}$}
\put(8,3){\oval(2,1)}
\put(7.5,2.8){$X_{1}$}
\put(12,1){\oval(2,1)}
\put(11.1,0.8){$Y\setminus Y_{1}$}
\put(12,3){\oval(2,1)}
\put(11.5,2.8){$Y_{1}$}

\put(9,3){\line(1,0){2.0}}

\put(8,1.6){\circle*{0.1}}
\put(8,1.8){\circle*{0.1}}
\put(8,2){\circle*{0.1}}
\put(8,2.2){\circle*{0.1}}
\put(8,2.4){\circle*{0.1}}

\put(12,1.6){\circle*{0.1}}
\put(12,1.8){\circle*{0.1}}
\put(12,2){\circle*{0.1}}
\put(12,2.2){\circle*{0.1}}
\put(12,2.4){\circle*{0.1}}

\put(6,-0.5){ Figure 1 : A split in a graph}

\end{picture}
\end{center}

%\caption{A split in a graph   }
\end{figure}

%%%%%%%%%%%%%%%%%%%%%%%%%%%%%%%%%%%%%%%%%%%%%%%%%%%%%%%%%%%%%%%%%%%%%%%%%%%%%%%%%%%%%%%%%%%%%
%%%%%%%%%%%%%%%%%%%%%%%%%%%%%%%%%%%%%%%%%%%%%%%%%%%%%%%%%%%%%%%%%%%%%%%%%%%%%%%%%%%%%%%%%%

%%%%%%%%%%%%%%%%%%%%%%%%%%%%%%%%%%%%%%%%%%%%%%%%%%%%%%%%%%%%%%%%%%%%%%%%%%%%%%%%%%%%%%%%%%
%%%%%%%%%%%%%%%%%%%%%%%%%%%%%%%   Fig 2   A bijoin in a graph and in a tournament   %%%%%%%%%%%%%%%%%%%%%%%%%%%%%%%%%%%%%
%%%%%%%%%%%%%%%%%%%%%%%%%%%%%%%%%%%%%%%%%%%%%%%%%%%%%%%%%%%%%%%%%%%%%%%%%%%%%%%%%%%%%%%%%%%%
\begin{figure}[h]
\begin{center}
\setlength{\unitlength}{0.7cm}
\begin{picture}(30,3)
\thicklines

\put(3,1){\oval(2,1)}
\put(2.1,0.8){$X\setminus X_{1}$}
\put(3,3){\oval(2,1)}
\put(2.5,2.8){$X_{1}$}
\put(7,1){\oval(2,1)}
\put(6.1,0.8){$Y\setminus Y_{1}$}
\put(7,3){\oval(2,1)}
\put(6.5,2.8){$Y_{1}$}
\put(4,3){\line(1,0){2.0}}
\put(4,1){\line(1,0){2.0}}

\put(3,1.6){\circle*{0.1}}
\put(3,1.8){\circle*{0.1}}
\put(3,2){\circle*{0.1}}
\put(3,2.2){\circle*{0.1}}
\put(3,2.4){\circle*{0.1}}

\put(7,1.6){\circle*{0.1}}
\put(7,1.8){\circle*{0.1}}
\put(7,2){\circle*{0.1}}
\put(7,2.2){\circle*{0.1}}
\put(7,2.4){\circle*{0.1}}

%%%%%%%%%%%%%%%%%%%%%%%
\put(13,1){\oval(2,1)}
\put(12.1,0.8){$X\setminus X_{1}$}
\put(13,3){\oval(2,1)}
\put(12.5,2.8){$X_{1}$}
\put(17,1){\oval(2,1)}
\put(16.5,2.8){$Y_{1}$}
\put(17,3){\oval(2,1)}
\put(16.1,0.8){$Y\setminus Y_{1}$}
\put(14,1){\vector(1,0){2.0}}
\put(14,3){\vector(1,0){2.0}}

\put(16.1,2.6){\vector(-2,-1){2.3}}
\put(16.1,1.4){\vector(-2,1){2.3}}

\put(13,1.6){\circle*{0.1}}
\put(13,1.8){\circle*{0.1}}
\put(13,2){\circle*{0.1}}
\put(13,2.2){\circle*{0.1}}
\put(13,2.4){\circle*{0.1}}

\put(17,1.6){\circle*{0.1}}
\put(17,1.8){\circle*{0.1}}
\put(17,2){\circle*{0.1}}
\put(17,2.2){\circle*{0.1}}
\put(17,2.4){\circle*{0.1}}

\put(3.5,-0.5){Figure $2$ : A bi-join in a graph and in a tournament  }

\end{picture}
\end{center}
%\caption{A bi-join in a graph and in a tournament  }
\end{figure}

%%%%%%%%%%%%%%%%%%%%%%%%%%%%%%%%%%%%%%%%%%%%%%%%%%%%%%%%%%%%%%%%%%%%%%%%%%%%%%%%%%%%%%%%%%%%%
%%%%%%%%%%%%%%%%%%%%%%%%%%%%%%%%%%%%%%%%%%%%%%%%%%%%%%%%%%%%%%%%%%%%%%%%%%%%%%%%%%%%%%%%%%

Bipartitive families are a general tool for studying both split decomposition
and bi-join decomposition. They were introduced by Montgolfier  \cite{Mon03} as follows. Let $V$ be a nonempty set. Two bipartitions $\left \{  X,Y\right \}  $ and
$\left \{  X^{\prime},Y^{\prime}\right \}  $ of $V$ \emph{overlap} if $X\cap Y$,
$X\cap Y^{^{\prime}}$, $X^{^{\prime}}\cap Y$ and $X^{^{\prime}}\cap
Y^{^{\prime}}$ are nonempty. A family $\mathcal{F}$ of bipartitions of $V$ is
\emph{weakly bipartitive} if:

\begin{description}
\item[Q1)] for all $v$ $\in$ $V$, $\left \{  \left \{  v\right \}  ,V\setminus
\left \{  v\right \}  \right \}  $ is in $\mathcal{F}$.

\item[Q2)] for all $\left \{  X,Y\right \}  $ and $\left \{  X^{\prime}%
,Y^{\prime}\right \}  $ in $\mathcal{F}$ such that $\left \{  X,Y\right \}  $
overlaps $\left \{  X^{\prime},Y^{\prime}\right \}  $, the four bipartitions
$\left \{  X\cap X^{\prime},Y\cup Y^{\prime}\right \}  $, $\left \{  X\cap
Y^{^{\prime}},Y\cup X^{^{\prime}}\right \}  $, $\left \{  Y\cap X^{\prime},X\cup
Y^{\prime}\right \}  $ and $\left \{  Y\cap Y^{\prime},X\cup X^{\prime}\right \}
$ are in $\mathcal{F}$.
\end{description}

A weakly bipartitive family $\mathcal{F}$ is \emph{bipartitive} if
it satisfies the following additional condition:

\begin{description}
\item[Q3)] for all $\left \{  X,Y\right \}  $ and$\left \{  X^{\prime},Y^{\prime
}\right \}  $ which overlap in $\mathcal{F}$, $\left \{  X\Delta X^{\prime
},X\Delta Y^{\prime}\right \}  $ is in $\mathcal{F}$.
\end{description}

Cunningham \cite{Cun82} proved that the family of splits of a connected graph
is bipartitive. The same result was obtained for the family of bi-joins of a
graph by Montgolfier \cite{Mon03}. For tournaments, the family of bi-joins is
only weakly bipartitive.

We will present now another important example of weakly bipartitive family which comes from the works of
\ Hartfiel and Loewy \cite{HL} and of Loewy \cite{Lw86}. Let $A=[a_{ij}]_{1\leq i,j\leq \text{ }n}$ be a
$n\times n$ matrix with entries in a field $\mathbb{K}$ and let $X,Y$\ be two
nonempty subsets of $\left[  n\right]  $ (where $\left[  n\right]  :=\left \{
1,\ldots,n\right \}  $). We denote by $A[X,Y]$ the submatrix of $A$ having row
indices in $X$ and column indices in $Y$. The matrix $A$
is\emph{\ irreducible} if for any proper subset $X$ of $\left[  n\right]  $,
both of matrices $A[X,\left[  n\right]  \setminus X]$ and $A[\left[  n\right]
\setminus X,X]$ are nonzero. An \emph{HL-bipartition} of $A$ is a partition
$\left \{  X,Y\right \}  $ of $\left[  n\right]  $ such that both of matrices
$A\left[  X,Y\right]  $ and $A\left[  Y,X\right]  $ have rank at most $1$. The concept of HL-bipartitions is equivalent to that of HL-clan \cite{BOCH}. In the case when $A$ is irreducible, the family of its HL-bipartitions is weakly bipartitive (see  Lemma 1 of \cite{Lw86}). 

Splits and bi-joins can be interpreted in terms of  HL-bipartitions. More precisely, we will prove in the next section that the splits (resp. the bi-joins) of an
undirected simple graph $G$ with vertex set $\left[  n\right]  $, are exactly
the HL-bipartitions of its adjacency matrix (resp. Seidel adjacency matrix).
Likewise, the bi-joins of a tournament $T$ with vertex set $\left[  n\right]
$ are the HL-bipartitions of its Seidel adjacency matrix.

 Throughout this paper, the family of HL-bipartitions  of a matrix  $A$ is denoted by $\mathcal{H}_{A}$. Our main result is the following theorem.

\begin{thm}
\label{Theoremeprincipal} If $A$ is a symmetric and irreducible $n\times n$
matrix over a field $\mathbb{K}$ then $\mathcal{H}_{A}$ is bipartitive.
Conversely, if $\mathcal{F}$ is a weakly bipartitive family of $\left[
n\right]  $ then there exists an irreducible matrix $A$ with entries in
$\{-1,0,1\}$ such that $\mathcal{F}=\mathcal{H}_{A}$. In the particular case
when $\mathcal{F}$ is bipartitive, the matrix $A$ can be chosen symmetric.
\end{thm}

%%%%%%%%%%%%%%%%%%%%%%%%%%%%%%%%%%%%%%%%%%%%%%%%%%%%%%%%%%%%%%%%%%%%%%%%%%%%%%%%%%%%
%%%%%%%%%%%%%%%%%%%%%%%%%%%%%%%%%%%    Splits, bi-joins and HL-bipartitions       %%%%%%%%%%%%%%%%%%%%%%%%%%%%%%%%%%%%%
%%%%%%%%%%%%%%%%%%%%%%%%%%%%%%%%%%%%%%%%%%%%%%%%%%%%%%%%%%%%%%%%%%%%%%%%%%%%%%%%%%%%%%

\section{Splits, bi-joins and HL-bipartitions}

Let $G$ be a graph with $n$ vertices $v_{1},...,v_{n}$. The \emph{adjacency}
\emph{matrix} of $G$ is the $n\times n$ real symmetric matrix $A(G)=[a_{ij}%
]_{1\leq i,j\leq \text{ }n}$ where $a_{ij}=1$ if $\left \{  v_{i},v_{j}\right \}
$ is an edge of $G$ and $a_{ij}=0$ otherwise. The \emph{Seidel adjacency}
\emph{matrix} of $G$ is the $n\times n$ symmetric matrix $S(G)=[s_{ij}]_{1\leq
i,j\leq \text{ }n}$ in which $s_{ij}=0$ if $i=j$ and otherwise is $-1$ if
$\left \{  v_{i},v_{j}\right \}  $ is an edge, $+1$ if it is not. The Seidel
matrix was introduced by Van Lint and Seidel \cite{Seidel}. Adjacency matrix
and Seidel matrix for a tournament are defined in the same way.

The following Proposition gives a description of splits and bi-joins in terms
of HL-bipartitions.

\begin{prop}
\label{HLand bijoinsplit} Let $G$ be a graph with vertex set $[n]$ let
$\left \{  X,Y\right \}  $ be a bipartition of $[n]$. Then

\begin{description}
\item[i)] $\{X,Y\}$ is a split of $G$ if and only if $\left \{  X,Y\right \}  $
is an HL-bipartition of $A(G)$.

\item[ii)] $\{X,Y\}$ is a bi-join of $G$ if and only if $\left \{  X,Y\right \}
$ is an HL-bipartition of $S(G)$.
\end{description}
\end{prop}

\begin{proof}
For positive integers $r$ and $s$, we denote by $0_{r,s}$ the $r\times s$ zero
matrix and by $J_{r,s}$ the $r\times s$ matrix of ones.

\begin{description}
\item[i)] Let $\left \vert X\right \vert :=p$ and $\left \vert Y\right \vert :=q$.
It is easy to see that $\{X,Y\}$ is a split of $G$ if and only if we can
reorder rows and columns of $A(G)\left[  X,Y\right]  $ so that the resulting
matrix is $0_{p,q}$, $J_{p,q}$ or one of the following matrices%
\[%
\begin{array}
[c]{ccc}%
\left(
\begin{tabular}
[c]{c|c}%
$J_{r,s}$ & $0_{r,q-s}$\\ \hline
$0_{p-r,s}$ & $0_{p-r,q-s}$%
\end{tabular}
\  \  \right)  & \left(
\begin{tabular}
[c]{c}%
$J_{r,q}$\\ \hline
$0_{p-r,q}$%
\end{tabular}
\  \  \right)  & \left(
\begin{tabular}
[c]{c|c}%
$J_{p,s}$ & $0_{p,q-s}$%
\end{tabular}
\right)
\end{array}
\]
These are the only possible forms (up to permutation of rows and columns) of a
$p\times q$ $(0,1)$-matrices having rank at most $1$.

\item[ii)] The argument is the same as in i). It suffices to check that
$\{X,Y\}$ is a bi-join of $G$ if and only if we can reorder rows and columns of
$S(G)\left[  X,Y\right]  $ so that the resulting matrix is $J_{p,q}$,
$-J_{p,q}$ or one of the following matrices:%
\[%
\begin{array}
[c]{ccc}%
\left(
\begin{tabular}
[c]{c|c}%
$J_{r,s}$ & $-J_{r,q-s}$\\ \hline
$-J_{p-r,s}$ & $J_{p-r,q-s}$%
\end{tabular}
\right)  & \left(
\begin{tabular}
[c]{c}%
$J_{r,q}$\\ \hline
$-J_{p-r,q}$%
\end{tabular}
\  \  \  \right)  & \left(
\begin{tabular}
[c]{c|c}%
$J_{p,s}$ & $-J_{p,q-s}$%
\end{tabular}
\right)
\end{array}
\]

\end{description}
\end{proof}

The results of Cunningham and Montgolfier mentioned in the introduction can be
deduced from the first assertion of our main theorem and the previous proposition.

A similar result of Proposition \ref{HLand bijoinsplit} holds for tournaments.
More precisely, we have the following.

\begin{prop}
Let $T$ be a tournament with vertex set $[n]$ and let $\left \{  X,Y\right \}  $
be a bipartition of $[n]$. Then $\{X,Y\}$ is a bi-join of $T$ if and only if
$\left \{  X,Y\right \}  $ is an HL-bipartition of $S(T)$.
\end{prop}

\section{Clans of $l2$-structures and their relationship with HL-bipartitions}

Let $V$ be a nonempty set and let $\widehat{V}^{2}:=\left \{  \left(
x,y\right)  /x\neq y\in V\right \}  $. Following \cite{ehr roz livre} a
\emph{labelled} $2$-\emph{structure }on $V$, or a $l2$-\emph{structure}, for
short, is a function $g$ from $\widehat{V}^{2}$ to a set\ of \emph{labels}
$\mathcal{C}$. With each subset  $X$ of $V$ associate the
$l2$-\emph{substructure} $g[X]$ of $g$ \emph{induced} by $X$ defined on $X$ by
$g[X](x,y):=g(x,y)$ for any $x\neq y\in X$. A  $l2$-structure $g$ on a set
$V$ is\emph{ symmetric }if $g(x,y)=g(y,x)$ for for every $x\neq y\in V$.

Let $g$ be a $l2$-structure on $\left[  n\right]  $ whose set of labels is a
field $\mathbb{K}$. We associate to $g$ the $n\times n$ matrix $M(g)=[m_{ij}%
]_{1\leq i,j\leq \text{ }n}$ in which $m_{ij}=0$ if $i=j$ and $m_{ij}=g\left(
v_{i},v_{j}\right)  $ otherwise. Conversely, let $A=[a_{ij}]_{1\leq
i,j\leq \text{ }n}$ be a matrix with entries in a field $\mathbb{K}$. We
associated to $A$ the $l2$-structure $g_{A}$ on $\left[  n\right]  $ and set
of labels $\mathbb{K}$ such that $g_{A}(i,j)=a_{ij}$ for $i\neq j$ $\in$
$\left[  n\right]  $.

Given a $l2$-structure $g$ on $V$, a subset $X$ of $V$ is a \emph{clan}
(\cite{ehr roz livre}, Subsection 3.2) of $g$ if for any $a,b\in$ $X$ and
$x\in B\setminus X$, we have $g(a,x)=g(b,x)$ and $g(x,a)=g(x,b)$.

\begin{rem}
\label{rema}
$\left.  {}\right.  $
\begin{description}
\item[i)] Graphs and tournaments can be seen as special classes $l2$%
-structure. Moreover, the notion of clan generalizes that of module.

\item[ii)] let $A$ be a matrix. if $I$ is a proper clan of $g_{A}$ then
$\left \{  I,\left[  n\right]  \setminus I\right \}  $ is an HL-bipartition of
$A$.
\end{description}
\end{rem}

The following Proposition appears in another form in \cite{BOCH} (see Lemma 2.2). It
describes the HL-bipartitions of a particular type of matrices called
\emph{normalized} matrices. Let $A=[a_{ij}]_{1\leq i,j\leq \text{ }n}$ be a
matrix and let $v\in \left[  n\right]  $. We say that $A$ is $v$%
-\emph{normalized }if $a_{vj}=a_{jv}=1$ for every $j\in \left[  n\right]
\setminus \left \{  v\right \}  $.

\begin{prop}
\label{BCgeneral} Let $A=[a_{ij}]_{1\leq i,j\leq \text{ }n}$ be a
$v$-normalized matrix for some $v\in \left[  n\right]  $ and let $I\subseteq
\left[  n\right]  \setminus \left \{  h\right \}  $. Then $\left \{  I,\left[
n\right]  \setminus I\right \}  $ is an HL-bipartition of $A$ if and only if
$I$ is a clan of $g_{A}\left[  \left[  n\right]  \setminus \left \{  v\right \}
\right]  $.
\end{prop}

\begin{proof}
In order to prove the necessary condition, let $i,j\in I$ and $k\in(\left[
n\right]  \setminus \left \{  v\right \}  )\setminus I$. Since $\left \{
I,\left[  n\right]  \setminus I\right \}  $ is an HL-bipartition of $A$, both
of matrices $A\left[  \left[  n\right]  \setminus I,I\right]  $ and $A\left[
I,\left[  n\right]  \setminus I\right]  $ have rank at most $1$. It follows
that $\det(A[\{v,k\},\{i,j\}])=\det(A[\{i,j\},\{v,k\}])=0$ and so
$g(k,i)=a_{ki}=a_{kj}=g(k,j)$ and $g(i,k)=a_{ik}=a_{jk}=g(j,k)$. We conclude
that $I$ is clan of $g_{A}(\left[  n\right]  \setminus \left \{  h\right \}  )$.
Conversely, let $I$ be a clan of $g_{A}\left[  \left[  n\right]
\setminus \left \{  v\right \}  \right]  $. Since $A$ is $v$-normalized, $I$ is a
clan of $g_{A}$ and then, by Remark \ref{rema}, $\left \{  I,\left[  n\right]
\setminus I\right \}  $ is an HL-bipartition of $A$.
\end{proof}

Let $V$ be a nonempty set $V$ and let $g$ be a $l2$-structure on $V$. We denote by $Cl(g)$ the family of
nonempty clans of $g$. This family satisfies the following well-known
properties (see, for example, Subsection 3.3 of \cite{ehr roz livre}).

\begin{description}
\item[P1)] $V\in P$, $\emptyset \notin Cl(g)$ and for all $v\in V$ , $\{v\} \in
Cl(g)$;

\item[P2)] Given $X,Y\in Cl(g)$; if $X$ and $Y$ overlap, that is $X\cap Y,$
$X\setminus Y$ and $Y\setminus X$ are all nonempty, then $X\cap Y\in Cl(g)$,
$X\setminus Y\in Cl(g)$, $Y\setminus X\in Cl(g)$ and $X\cup Y\in Cl(g)$.
\end{description}

Moreover, if $g$ is symmetric then $Cl(g)$ satisfies the additional property:

\begin{description}
\item[P3)] Given $X,Y\in Cl(g)$; if $X$ and $Y$ overlap then $X\bigtriangleup
Y=(X\setminus Y)\cup(Y\setminus X)\in Cl(g)$.
\end{description}

Let  $\mathcal{P}$ be a family of subsets of $V$.
We say that $\mathcal{P}$ is \emph{weakly partitive }if \textbf{P1} and
\textbf{P2} hold. If also \textbf{P3} holds, we say that $\mathcal{P}$ is
\emph{partitive. }Partitive and weakly partitive families were introduced in
\cite{CHM81}. They are closely related to partitive families as shown in the
following lemma.

\begin{lem}
\label{equipartibipartiRao} Let $\mathcal{B}$ be
a family of bipartitions of $V$ and let $v\in V$. We denote by $\mathcal{P}$ the family of
subsets $X$ of $V\setminus \left \{  v\right \}  $ such that $\{X,V\setminus X\}
\in \mathcal{B}$. Then $\mathcal{B}$ is weakly bipartitive (resp.bipartitive)
if and only if $\mathcal{P}$ is weakly partitive (resp. partitive).
\end{lem}

The next Theorem of gives relationship between weakly partitive family and
clans family.

\begin{thm}
 \label{thmPille}Let $\mathcal{P}$ be a weakly partitive
family on  $V$, then there exists an $l2$-structure $g$ on $V$ with
labels in a set of size at most $3$ such that $\mathcal{P=}Cl(g)$. Moreover if
$\mathcal{P}$ is partitive family on a set $V$, then $g$ can be chosen symmetric.
\end{thm}

The first part of this theorem was proved by Ehrenfeucht, Harju, and Rozenberg
(see \cite{ehr roz livre}, Theorem 5.7), and later by Ille and Woodrow
\cite{Ille09}. As noted by Ille \cite{Ille2016}, the method given in
\cite{Ille09} can also be used to prove the second part.

\section{Proof of main theorem}

We start with the following result.

\begin{prop}

 \label{key coeff non nul}Let $A=[a_{ij}]_{1\leq i,j\leq \text{ }n}$ be an
irreducible $n\times n$ matrix with entries in a field $\mathbb{K}$. Then for
every $v\in \left[  n\right]  $ there is a $v$-normalized matrix $\widehat{A}$
with non zero entries in a field $\widehat{\mathbb{K}}$ containing
$\mathbb{K}$ such that $A$ and $\widehat{A}$ have the same HL-bipartitions.
Moreover, if $A$ is symmetric then $\widehat{A}$ can be chosen symmetric
\end{prop}

For the proof of this proposition, we use the following lemma.

\begin{lem}
\label{irreandinvers} Let $A=[a_{ij}]_{1\leq i,j\leq \text{ }n}$ be a
irreducible matrix. Let $x_{1},x_{2},\ldots,x_{n}$ be (independent)
indeterminates, $\chi=diag(x_{1},x_{2},\ldots,x_{n})$. Then we have the
following statements:

\begin{description}
\item[i)] the matrix $A+\chi$ is invertible in $\mathbb{K}(x_{1},x_{2}%
,\ldots,x_{n})$.

\item[ii)] all entries of $(A+\chi)^{-1}$ are nonzero.

\item[iii)] $A$, $A+\chi$ and $(A+\chi)^{-1}$ have the same HL-bipartitions.
\end{description}
\end{lem}

For assertions i) and ii) of this lemma, see Theorem 1 of \cite{HL}. The third assertion is a direct consequence of the following Proposition. 

\begin{prop}
\cite{HL} Let $T$ be an invertible matrix over $\mathbb{K}$, and
suppose it has a block form

$T=\left(
\begin{array}
[c]{cc}%
T_{11} & T_{12}\\
T_{21} & T_{22}%
\end{array}
\right)  $

where $T_{11}$ is an invertible $k\times k$ matrix. Let $W=T^{-1}$, and
partition $W$ conformably with $T$, so

$W=\left(
\begin{array}
[c]{cc}%
W_{11} & W_{12}\\
W_{21} & W_{22}%
\end{array}
\right)  $

Then $rank(W_{12})=rank(T_{12})~$and $rank(W_{21})=rank(T_{21})$.
\end{prop}

\begin{pop}
 We will use the notations of Lemma
\ref{irreandinvers}. Let $(A+\chi)^{-1}:=[b_{ij}]_{1\leq i,j\leq \text{ }n}$,
$D:=[d_{i}]_{1\leq i\leq \text{ }n}$ and $D^{\prime}:=[d_{i}^{\prime}]_{1\leq
i\leq \text{ }n}$ where $d_{i}=\frac{1}{b_{iv}}$ , $d_{i}^{\prime}=\frac
{1}{b_{vi}}$ for $i\neq v$ and $d_{v}=d_{v}^{\prime}=1$. Clearly, the matrix
$\widehat{A}:=D(A+X)^{-1}D^{\prime}$ is $v$-normalized \ and its entries are
in $\widehat{\mathbb{K}}=\mathbb{K}(x_{1},x_{2},\ldots,x_{n})$. Moreover, if
$A$ is symmetric then $A+\chi$ and $(A+\chi)^{-1}$are also symmetric. It
follows that $D=D^{\prime}$ and hence $\widehat{A}$ is symmetric. We conclude
by applying iii) of Lemma \ref{irreandinvers} and the following lemma.
\end{pop}

\begin{lem}
\label{memHL-clan}Let $M$ be a $n\times n$ matrix and let $D_{1}$, $D_{2}$ be
two $n\times n$ diagonal and invertible matrices. Then, the matrices $M$ and
$D_{1}MD_{2}$ have the same HL-bipartitions.
\end{lem}

\begin{proof}
Let $X,Y$ be two subset of $\left[  n\right]  $. We have the following
equalities:
\begin{align*}
(D_{1}MD_{2})\left[  X,Y\right]   &  =(D_{1}\left[  X\right]  )(M\left[
X,Y\right]  )(D_{2}\left[  Y\right]  )\\
(D_{1}MD_{2})\left[  Y,X\right]   &  =(D_{1}\left[  Y\right]  )(M\left[
Y,X\right]  )(D_{2}\left[  X\right]  )
\end{align*}

It follows $(D_{1}MD_{2})\left[  X,Y\right]  $ and $(M\left[  X,Y\right]  )$
(resp. $(D_{1}MD_{2})\left[  Y,X\right]  $ and $M\left[  Y,X\right]  $) have
the same rank because the matrices $D_{1}\left[  X\right]  $, $D_{2}\left[
X\right]  $, $D_{1}\left[  Y\right]  $ and $D_{2}\left[  Y\right]  $ are
invertible. Thus, $\left \{  X,Y\right \}  $ is an HL-bipartition of $M$ if and
only if it is one for $D_{1}MD_{2}$.
\end{proof}

\begin{pot}
The fact that $\mathcal{H}_{A}$ is weakly bipartitive follows from Lemma 1
of \cite{Lw86}. To complete the
proof it suffices to check \ that $\mathcal{H}_{A}$ satisfies the condition
\textbf{Q3}.\textbf{ }For this,\textbf{ }let $\left \{  X,Y\right \}  $,
$\left \{  X^{\prime},Y^{\prime}\right \}  \in \mathcal{H}_{A}$ which overlap.
Then$\  \left[  n\right]  \setminus(X\cup X^{\prime})=Y\cap Y^{\prime}%
\  \neq \emptyset \ $. Let $i\  \in \left[  n\right]  \  \setminus(X\cup X^{\prime
})$. By Proposition \ref{key coeff non nul}, there is a symmetric and
$i$-normalized matrix $\widehat{A}$ such that $\  \mathcal{H}_{A}%
=\mathcal{H}_{\widehat{A}}$. So it suffices to prove that $\left \{  X\Delta
X^{\prime},X\Delta Y^{\prime}\right \}  \in \mathcal{H}_{\widehat{A}}$. By the
choice of $i$, we have $i\notin X$ and $i\notin X^{\prime}$ and then by Lemma
\ref{BCgeneral} $X$ and $X^{\prime}$ are clans of $g_{\widehat{A}}\left[
\left[  n\right]  \setminus \left \{  i\right \}  \right]  $. Moreover, $X$ and
$X^{\prime}$ overlap because $\left \{  X,Y\right \}  $, $\left \{  X^{\prime
},Y^{\prime}\right \}  \in \mathcal{H}_{A}$  overlap. Now, since $\widehat{A}$
is symmetric, $g_{\widehat{A}}\left[  \left[  n\right]  \setminus \left \{
i\right \}  \right]  $ is symmetric and then by \textbf{P3}, $X\Delta
X^{^{\prime}}$ is a clan of $g_{\widehat{A}}\left[  \left[  n\right]
\setminus \left \{  i\right \}  \right]  $. By applying again Lemma
\ref{BCgeneral}, we deduce that $\left \{  X\bigtriangleup X^{\prime
},X\bigtriangleup Y^{\prime}\right \}  \in \mathcal{H}_{\widehat{A}}$.

Conversely, let $\mathcal{F}$ be a weakly bipartitive family on a set $[n]$. We
will construct an irreducible matrix $A$ with entries in $\{-1,0,1\}$ such
that $\mathcal{F}=$ $\mathcal{H}_{A}$. From Lemma \ref{equipartibipartiRao}
the family $\mathcal{P}:=\left \{  X\subseteq \left[  n-1\right]
:\{X,[n]\setminus X\} \in \mathcal{F}\right \}  $ is weakly partitive, then by
Theorem \ref{thmPille}, there exists an $l2$-structure $g$ on $\left[
n-1\right]  $ with labels in $\{-1,0,1\}$ such that $\mathcal{P=}Cl(g)$.
Consider the following matrix
\[
A=\left(
\begin{tabular}
[c]{ccc|c}
&  &  & $1$\\
& $M(g)$ &  & $\vdots$\\
&  &  & $1$\\ \hline
$1$ & $\cdots$ & $1$ & $0$%
\end{tabular}
\  \  \right)
\]

Clearly, this matrix is $n$-normalized and then it is irreducible. To prove
that $\mathcal{F}=$ $\mathcal{H}_{A}$, let $\{X,\left[  n\right]  \setminus
X\}$ be a bipartition of $\left[  n\right]  $ and assume for example that
$n\notin X$. By Lemma \ref{BCgeneral}, $\{X,\left[  n\right]  \setminus
X\} \in \mathcal{H}_{A}$ if and only if $X$ is a clan of $g_{A}\left[
1,\ldots,n-1\right]  =g$. Then $\{X,\left[  n\right]  \setminus X\} \in
\mathcal{H}_{A}$ if and only if $X\in$ $\mathcal{P}$ or equivalently
$\{X,[n]\setminus X\} \in \mathcal{F}$ because $\mathcal{P=}Cl(g)$.

Now if $\mathcal{F}$ is bipartitive, then the family $\mathcal{P}:=\left \{
X\subseteq \left[  n-1\right]  :\{X,[n]\setminus X\} \in \mathcal{F}\right \}  $
is partitive. By Theorem \ref{thmPille}, we can choose $g$ symmetric, which
implies that $A$ is symmetric.
\qed\end{pot}

\end{document}